\documentclass[12pt,twoside,reqno]{amsart}
\usepackage{amsfonts}
\usepackage{url}
\usepackage{amsmath,amscd,amsthm,amssymb}
\usepackage{latexsym}
\newtheorem{theorem}{Theorem}
\newtheorem*{theorem*}{Theorem}

\newtheorem{corollary}[theorem]{Corollary}

\newtheorem{lemma}[theorem]{Lemma}
\newtheorem{remark}[theorem]{Remark}
\newtheorem*{claim*}{Claim}
\def   \abs#1{\left\lvert {#1}\right\rvert}
\def \paren#1{\left(      {#1}\right)}
\def \sumstack#1{\sum_{\substack{#1}}}
\def\Np{\mathrm{N}\mathfrak{p}}
\def\Ot{\widetilde{O}}

\newcommand{\Z}{\mathbb{Z}}
\newcommand{\Q}{\mathbb{Q}}
\newcommand{\R}{\mathbb{R}}
\newcommand{\C}{\mathbb{C}}

\newcommand{\p}{\mathfrak{p}}
\newcommand{\e}{e}
\newcommand{\re}{\mathrm{Re}}
\newcommand{\im}{\mathrm{Im}}
\newcommand{\ie}{{\it{i.e.\ }}}
\newcommand{\cl}{{\mathcal{C}}\ell_K}
\newcommand{\B}{{\mathcal{B}}}

\begin{document}
\title[Computing the residue of the Dedekind zeta function]{Computing the residue \\ of the Dedekind zeta function}

\author{Karim Belabas}
\address{Universit\'e de Bordeaux 1, Math\'ematiques Pures, 351 Cours de la
Lib\'eration, F-33405~Talence cedex, France}
\email{Karim.Belabas@math.u-bordeaux1.fr}

\author{Eduardo Friedman}
\address{Departamento de Matem\'atica,
 Universidad de Chile, Casilla 653, Santiago, Chile}
\email{friedman@uchile.cl}

\thanks{The first author was supported by the  ANR projects ALGOL
(\texttt{07-BLAN-0248}) and PEACE (\texttt{ANR-12-BS01-0010-01}).}

\thanks{The second author was partially supported by the Chilean Programa
Iniciativa Cient\'{\i}fica Milenio grant ICM P07-027-F and Fondecyt grant
1110277.}

\begin{abstract} Assuming the Generalized Riemann Hypothesis, Bach has shown
  that one can calculate the residue of the Dedekind zeta function of a
  number field $K$ by a clever use of  the splitting of primes $p<X$, with an
  error asymptotically bounded by  $8.33\log \Delta_K/(\sqrt{X}\log X)$,
  where $\Delta_K$ is the absolute value of the discriminant of $K$.  Guided
  by Weil's explicit formula and still assuming GRH, we make a different use
  of the splitting of  primes and thereby improve Bach's constant to $2.33$.
  This results in substantial speeding of one part of Buchmann's class group
  algorithm.
\end{abstract}

\maketitle
\section{Introduction}

Given a number field $K$, Buchmann's algorithm~\cite{Buch:subexp} computes the
ideal class group $\cl$ and units $U(K)$. It uses an index calculus strategy
which requires a \emph{factor base} $\B$ of prime ideals generating $\cl$,
and a halting criterion based on a computed approximation $\widehat{hR}$ of
the product of the class number $h$ by the regulator~$R$. Indeed, it produces
elements in the kernel $\Lambda$ of the natural surjective map
$\Z^\B\twoheadrightarrow \cl$ by factoring principal ideals $(\alpha)$, then
proceeds to find dependencies between those, yielding pairs $\alpha,\alpha'$
generating the same principal ideals, \ie units $\alpha/\alpha'$. This 
gives a tentative class number $\hat{h}$ and a tentative regulator $\hat{R}$,
both integral multiples of $h$ and $R$, respectively. If we find
$\hat{h}\hat{R} < 2hR$, then $h=\hat{h}$ and $R=\hat{R}$, thereby halting the
algorithm.

Buchmann's algorithm requires two important inputs: 
\begin{itemize}
\item a factorbase $\B = \B(K)$ so that $\Z^{\B}\twoheadrightarrow \cl$,

\item an approximate value of $\log(hR)$, with a rigorous error
  term.\footnote{It suffices to make the error less than $\frac{1}{2}
    \log 2$.}
%    : assume that $hR \leq \hat{h}\hat{R}$ and $\abs{c - \log (hR)} <
%    \varepsilon$; then $\log (\hat{h}\hat{R}) - c > - \varepsilon$, and
%    $|\log (\hat{h}\hat{R}) - c|  < \log 2 - \varepsilon$ ensures
%    $\hat{h}\hat{R} < 2hR$. For $\varepsilon = \frac12\log 2$, the condition
%    $\hat{h}\hat{R} < 2hR$ is thus equivalent to $|\log (\hat{h}\hat{R}) - c|
%  < \varepsilon$.}
\end{itemize}
Assuming a suitable Generalized Riemann Hypothesis (GRH),
Bach~\cite{Bac:exp,Bac:eul}
showed how to choose a reasonably small $\B$ \emph{and} found an
approximation for $\log(hR)$ using averages of truncated Euler products.
Schoof~\cite{Sch:cla} had previously found a simpler approximation, but with
a worse error bound.

This paper is a companion to \cite{Bel:BDF}, where we improved
\emph{numerically} on Bach's first result (factorbase choice) using the
Poitou-Weil explicit formula~\cite{Poitou1976}. Our main aim here is to 
improve on Bach's second result. Let
\begin{align*}
B_K(X)&:=   \sumstack{\p,m\\ \Np^m <X}^{K-\Q}\frac{\log \Np}{\Np^{m/2}}
  \Big(\frac{\sqrt{X}\log X}{
\Np^{m/2}\log\Np^m}-1\Big),\\
f_K(X)&:= \frac{3\big(B_K(X)-B_K(X/9)\big)}{2\sqrt{X}\log(3X)},
\end{align*}
where in the definition of $B_K$ the sum is over all prime ideal powers $\p^m$
with absolute norm $\Np^m < X$ and the notation $\sum^{K-k}$ means that the
sum for $k$ is subtracted from the corresponding sum for $K$.

\begin{theorem}\label{Mainresult}
Let $K$ be a number field of degree $n>1$,
let $\kappa_K$ be the residue of the  Dedekind zeta $\zeta_K(s)$ at $s=1$,
and let $\Delta_K$ be the absolute value of the discriminant of $K$.
Assume  GRH, \ie that $\zeta_K(s)\neq 0$ and $\zeta_\Q(s)\neq 0$
whenever $\re(s)>\frac12$. Then, for any real $X \geq 69$,  the difference
$\abs{\log \,\kappa_K- f_K(X)}$ is bounded above by
$$
\frac{2.324\log \Delta_K}{\sqrt{X}\log(3X)}
\Bigg( \Big(1+
\frac{3.88}{\log(X/9)}\Big)\Big(1+\frac{2}{\sqrt{\log \Delta_K}}\Big)^2  +
\frac{4.26(n-1)}{\sqrt{X}\log \Delta_K} \Bigg).
$$
\end{theorem}
Bach's original result \cite[Lemma 4.7 and \S7]{Bac:eul}, also assuming GRH,
is of the form
$$
  \abs{ \log \,\kappa_K - g_K(X) }
  \le \frac{8.324\log \Delta_K}{\sqrt{X}\log(X/2)}
      \big(1 + E(\Delta_K, X)\big),
$$
where $g_K(X)$ is a function involving prime ideals of norm $\leq
X$ \big(different from $f_K(X)$\big), and $E(\Delta, X) \to
0$.\footnote{ Here $8.324 \approx \sqrt{2}\cdot \frac{2}{3} \cdot
  \paren{2^{3/2} + 6}$\; \cite[p.~22]{Bac:eul}. In comparing our result with
  Bach's, one should bear in mind that Bach's $x$ is our $X/2$, since $X$
bounds the biggest rational prime whose splitting must be computed.} Both
Bach's and our results show that choosing $X = O(\log \Delta_K / \log\log
\Delta_K)^2$ computes $\log(hR)$ with an error bounded by
$\frac{1}{2}\log 2$. Our better error bounds translate to a shorter list of
prime ideals, by an asymptotic factor of $(8.324/2.324)^2 \approx 12.8$, and
correspondingly faster computations for $\log \kappa_K$. In
section~\ref{se:examples}, we give tables comparing Schoof's, Bach's and
our method for various ranges of $\Delta_K$ and $[K:\Q]$.

\section{The explicit formula}
%%%
 Weil's explicit formula \cite{Wei:52}, as simplified by Poitou \cite{Poitou1976},  is the identity
\begin{align}
\sum_\rho \widehat{F}(\gamma_\rho) = \,&
  -2\sum_{\p,m} \frac{\log \Np}{\Np^{m/2}} F(m\log \Np)
   +4 \int_0^{\infty} \!\!F(x)\cosh(x/2)\,dx
\nonumber \\
& 
 + F(0)\Big(\log \Delta_K -n_K C - n_K\log(8\pi)-r_K\frac{\pi}2\Big)\label{Weil}\\
&+ n_K\int_0^{\infty} \frac{F(0)-F(x)}{2\sinh(x/2)} \,dx + r_K \int_0^{\infty} \frac{F(0)-F(x)}{2\cosh(x/2)}\,dx
  .\nonumber
\end{align}
Here  $K$ is a number field of degree $n_K=[K:\Q]$,  having exactly $r_K$
real embeddings, and $\Delta_K$ is the absolute value of its discriminant.
The auxiliary function $F\colon \R\to\C$ is assumed to be even,  and such
that for some $\varepsilon>0$, the function $ F(x)
\e^{(\frac12+\varepsilon)x}$ is of bounded variation and integrable over
$[0,+\infty)$. Also $\big(F(0)-F(x)\big)/x$ is assumed of bounded variation
on $[0,+\infty)$ and $F$ must be assigned  the average value at any jump
discontinuity. By $C$ we mean Euler's constant $0.5772\cdots$.  

The Fourier transform $\widehat{F} $ of $F$ on the left-hand side of
\eqref{Weil} is
\begin{equation}\label{FT}
\widehat{F}(\gamma) :=
  \int_{-\infty}^{+\infty} F(t) \e^{i t\gamma} \,dt.
\end{equation}
The sum of the $\widehat{F}(\gamma_\rho)$ runs over all nontrivial  zeroes
$\rho=\frac12+i\gamma_\rho$ of the Dedekind zeta function $\zeta_K(s)$, with
multiple zeroes repeated accordingly. The Riemann Hypothesis
(GRH) for $\zeta_K$ states that $\gamma_\rho\in\R$. Given our assumptions
on $F$, the sum over $\rho$ converges when understood as
$$
\lim_{R\to+\infty} \sum_{|\im(\rho)| < R}\widehat{F}(\gamma_\rho).
$$
In the sum on the right of \eqref{Weil}, $\p$ runs over all prime ideals of
(the ring of algebraic integers of) $K$, $m$ runs over all positive integers,
and the absolute norm of $\p$ is denoted by $\Np$.

If $K$ and $k$ are  number fields, on subtracting Weil's formula for $k$ from
\eqref{Weil}, we obtain the form we shall mostly use
\begin{multline}\label{eq:poitou1}
 \sum^{K-k}_{\rho} \widehat{F} (\gamma_\rho) =
-2 \sum^{K-k}_{\p, m} \frac{\log \Np}{\Np^{m/2}} F(m\log N\p)
 + F(0)L_{K/k} \\
+  (n_K-n_k)\int_0^{\infty} \frac{F(0)-F(t)}{2\sinh(t/2)}\,dt
+ (r_K-r_k)\int_0^{\infty} \frac{F(0)-F(t)}{2\cosh(t/2)}\,dt,
\end{multline}
where
$$
L_{K/k}:=  \log \paren{ \frac{\Delta_K}{\Delta_k} }
  - (n_K-n_k)\big(C + \log(8\pi)\big)
  - (r_K-r_k)\frac{\pi}{2} ,
$$

\section{The auxiliary function}
In this section we explain how our choice of auxiliary function $F=F_{s,X}$
is motivated by the form of the explicit formula and the need to avoid
bounding conditionally convergent expressions.

If $K$ and $k$ are number fields, the obvious path to computing   $$\kappa_{K/k}:=\lim_{s\to1}\frac{\zeta_K}{\zeta_k}(s)$$ 
 is via the Euler product
 $\zeta_K(s)=\prod_\p (1 - \Np^{-s})^{-1}$, \ie 
\begin{equation}\label{eq:LogZeta}
\log\frac{\zeta_K}{\zeta_k}(s)
  =-\sum_{\p}^{K-k} \log(1 - \Np^{-s})
  =\sum_{\p}^{K-k}\sum_{m=1}^\infty \frac{\Np^{-ms}}{m}
       \quad(\re(s)>1).
\end{equation}
A na\"ive attempt to approximate $\log\frac{\zeta_K}{\zeta_k}(s)$ by a
partial sum would therefore be
\begin{align}\label{eq:logres}
\log \frac{\zeta_K}{\zeta_k}(s)
-\sumstack{\p,m\\ \Np^m < X}^{K-k} \frac{\Np^{-ms}}{m}
=  \sum_{ \p,m }^{K-k}\log\Np\,\frac{H(\log \Np^m)}{\Np^{m/2}},
\end{align}
where (for $X$ not a prime power)
$$
H(t) = H_{s,X}(t) :=
\begin{cases}
g_s(t) & \text{if }  |t| \ge \log X  , \\
0 & \text{otherwise},
\end{cases}
$$
\noindent and where
\begin{equation}
  g_s(t) := \dfrac{\exp\paren{-h|t|}}{|t|},\quad h := s -\frac{1}{2}.
  \label{eq:gs}
\end{equation}
The explicit formula \eqref{eq:poitou1} gives an expression for
the right-hand side of \eqref{eq:logres}. Its most interesting term is
$  \sum_{\rho}^{K-k}\widehat{H}(\gamma_\rho)$.
While there is no simple closed expression for 
$\widehat{H}$, it is easy to write its leading term. After two integrations by
parts using
\begin{align}\label{eq:diffeq}
g_s'(t) = - \Big(h + \frac{1}{t} \Big)g_s(t), \quad  
g_s''(t) = \Big(h^2 + \frac{2ht + 2}{t^2} \Big)g_s(t),
            \end{align}
and setting  $T:=\log X>0$, we obtain 
\begin{align}\label{eq:FirstTry}
\widehat{H}(\gamma) =&-   g_s(T)  \bigg( \frac{2\gamma\sin(\gamma T)}{h^2+\gamma^2}
- \frac{2\big( h +\frac{1}{T}\big) \cos(\gamma T)}{h^2+\gamma^2} \bigg) \vspace{.2cm}
 \notag \\
   & \qquad -\frac{4}{h^2+\gamma^2} \int_T^{+\infty} \cos(\gamma t)
            g_s(t)\frac{   (ht  + 1)}{t^2} \,dt.
\end{align}
Even assuming GRH, the first term is highly unwelcome since we cannot control
\begin{equation}\label{eq:zerosum}
\sum_\rho^{K-k} \frac{2\gamma_\rho\sin(\gamma_\rho T)}
                     {h^2+\gamma_\rho^2}
\end{equation}
by its absolute value.\footnote{The rest of the terms are easily bounded
under GRH, as we shall see in the next section.} The simple identity
\begin{align}\label{eq:BadTerm}
\frac{\gamma \sin(\gamma T)}{h^2+\gamma^2}=  \frac{\sin(\gamma T)}{\gamma }
-\frac{h^2}{(h^2+\gamma^2)} \frac{\sin(\gamma T)}{\gamma },
\end{align}
shows that our troubles in \eqref{eq:zerosum} come from $\sum_\rho
\frac{2\sin(\gamma_\rho T)}{\gamma_\rho}$. Fortunately, this is just the term
that appears in the explicit formula when we use as auxiliary function the
step function
$$
\widetilde{H}(t)  :=
\begin{cases}
1 & \text{if $\abs{t} \le T$}, \\
0 & \text{otherwise}.
\end{cases}
$$
To cancel the bad term $\sin(\gamma T)/ \gamma$ in \eqref{eq:BadTerm} we must
therefore choose the auxiliary function to be $H(t)+g_s(T)\widetilde{H}(t)$.
Normalizing so that $F(0)=1$ leads to our  auxiliary function
\begin{equation}\label{eq:Aux}
F(t) = F_{s,X}(t) :=
\begin{cases}
1             & \text{if $|t| \le \log X$},\\
 f_{s,X}(t)   & \text{otherwise},
\end{cases}
\end{equation}
where
\begin{equation}
\label{eq:Auxf}
f_{s,X}(t) := \frac{g_s(t)}{g_s(T)}
                 =\frac{T}{|t|}e^{-h(|t|-T)}
\quad\big(h:=s-\textstyle{\frac12},\ T:=\log X\big).
\end{equation}
We shall see in the next section that this  choice of $F$ leads to a sum
$\sum_{\rho}^{K-k}\widehat{F}(\gamma_\rho)$
which can be controlled well under GRH.

Using \eqref{eq:FirstTry}, we obtain:
\begin{lemma}\label{Lemma:FT} For $\re(s)>\frac12$ and $ X>1$, let $T:=\log X$,
let $F_{s,X}$ be as  in \eqref{eq:Aux} and let $\widehat{F}_{s,X}$ be its
Fourier transform \eqref{FT}. Then, for $\gamma\in\R$ and notation as in
\eqref{eq:Auxf}, we have
\begin{multline*}
\widehat{F}_{s,X}(\gamma) =
 \frac{2h^2\sin(\gamma T)}{(h^2+\gamma^2)\gamma}+
 \frac{2\big(h+\frac{1}{T}\big)\cos(\gamma T)}{h^2+\gamma^2}
\\[0.2cm]
 - \frac{4}{(h^2+\gamma^2)} \int_T^{+\infty} \cos(\gamma t)
            f_{s,X}(t)\frac{(ht  + 1)}{t^2} \,dt.
\end{multline*}
\end{lemma}

\section{Proof of Theorem~\ref{Mainresult}}

We now apply Lemma~\ref{Lemma:FT} to the explicit formula.
\begin{lemma}\label{Partways} Let $K$ and $k$ be number fields such that the
Riemann Hypothesis holds for $\zeta_K$ and $\zeta_k$. Then, for
$\re(s)>\frac12, \ \, T:=\log X>0,$ and notation as in \eqref{eq:poitou1},
\eqref{eq:gs} and \eqref{eq:Auxf}, we have
\begin{align}\nonumber
  \frac{1}{g_s(T)}&
  \log \frac{\zeta_K}{\zeta_k}(s)
-   \sumstack{\p,m\\ \Np^m <X}^{K-k} \frac{\log \Np}{\Np^{m/2}}
  \big(f_{s,X}(m\log \Np)-1\big)
\\ \label{eq:Explicit1}
& =
  -h^2\sum^{K-k}_\rho
\frac{\sin (\gamma_\rho T)}{(h^2+\gamma_\rho^2)\gamma_\rho}
 \  - \ \Big(h+\frac{1}{T}\Big) \sum^{K-k}_\rho
\frac{\cos(\gamma_\rho T)}{h^2+\gamma_\rho^2}
 \\
&\ \ \  + \sum^{K-k}_\rho\frac{2}{ h^2+\gamma_\rho^2 }\int_{T}^{+\infty} \frac{ (ht + 1)}{t^2}  \cos(\gamma_\rho t)
            f_{s,X}(t)\,dt +\frac12L_{K/k} \nonumber\\
&\ \ \ + \frac{n_K-n_k}{2}\int_{T}^{\infty} \frac{1 - f_{s,X}(t)}{2\sinh(t/2)}\,dt
+ \frac{r_K-r_k}{2}\int_{T}^{\infty} \frac{1 - f_{s,X}(t)}{2\cosh(t/2)}\,dt.\nonumber
\end{align}
\end{lemma}
\noindent The branch of $\log \frac{\zeta_K}{\zeta_k}(s)$ in
\eqref{eq:Explicit1} is real for real $s>1$.
\begin{proof}
Assume first that $\re(s)>1$. Then the assumptions in the explicit formula
\eqref{eq:poitou1} apply to $F_{s,X}$  in \eqref{eq:Aux}, so we find
\begin{multline*}
 2 \sumstack{\p,m\\ \Np^m <X}^{K-k}
  \frac{\log \Np}{\Np^{m/2}} 
  \big(1 - f_{s,X}(m\log \Np)\big)
 + 2 \sum_{\p,m}^{K-k}  \frac{\log \Np}{\Np^{m/2}} 
    f_{s,X}(m\log \Np)\\
+ \sum^{K-k}_{\rho} \widehat{F}_{s,X}(\gamma_\rho)
\;=\; L_{K/k}
+ (n_K-n_k)\int_{T}^{\infty} \frac{1 - f_{s,X}(t)}{2\sinh(t/2)}\,dt\\
+ (r_K-r_k)\int_{T}^{\infty} \frac{1 - f_{s,X}(t)}{2\cosh(t/2)}\,dt.
\end{multline*}
Note  that $\big($cf.\ \eqref{eq:LogZeta}$\big)$,
\begin{align*}
 \sum_{\p,m}^{K-k} \frac{\log \Np}{\Np^{m/2}}
      f_{s,X}(m\log\Np) 
& =
\frac{1}{g_s(T)}
  \sum_{\p, m}^{K-k} \frac{\log \Np}{\Np^{m/2}}
  g_s(m\log\Np)\\
& =
\frac{1}{g_s(T)} \log  \frac{\zeta_K}{\zeta_k}(s).
\end{align*}
The lemma for $\re(s)>1$ now follows from   Lemma \ref{Lemma:FT}.

To obtain \eqref{eq:Explicit1} for $\re(s)>\frac12$ by analytic continuation,
note that GRH  implies $\gamma_\rho^2+h^2\not=0$ for $\re(s)>\frac12$, and
that  $\log \frac{\zeta_K}{\zeta_k}(s)$ is analytic in that half-plane. Hence
we only need to estimate for $\re(s)=\sigma>\frac12$,
\begin{align*}
\int_{T}^{+\infty}  \Big| \frac{ (ht + 1)}{t^2}  \cos(\gamma_\rho t)
            f_{s,X}(t) \Big|\,dt
& \le
 \frac{|h|T+1}{T^3g_\sigma(T)}\int_{T}^{+\infty}   \e^{-(\sigma-\frac12)t}\,dt
\\
& = \frac{|h|T+1}{T^2 (\sigma-\frac12)}.
 \end{align*}
\end{proof}

  Lemma \ref{Partways} nearly takes us to our goal  since $g_s(T)=1/(X^{s-\frac12}\log X)$ for $T=\log X$.
   Indeed, multiplying \eqref{eq:Explicit1}  by $g_s(T)$ and letting $\sigma=\re(s)>\frac12$, we see that to obtain
\begin{align*}%\label{eq:wish}
  \Big|\log \frac{\zeta_K}{\zeta_k}(s)
-  g_s(T) \sumstack{\p,m\\ \Np^m <X}^{K-k} \frac{\log \Np}{\Np^{m/2}}
  \big(f_{s,X}(m\log \Np)-1\big)\Big|<   \frac{c\log \Delta_K}{X^{\sigma-\frac12}\log X}
\end{align*}
it would suffice to bound the right-hand side of  \eqref{eq:Explicit1} by $c\log \Delta_K$. It is well-known
 (see Lemma \ref{Estimate}) that
$$
\sumstack{\rho\\ \zeta_K(\rho)=0} \frac{1}{h^2+\gamma_\rho^2}=\mathcal{O}\big(\log \Delta_K\big),
$$
Unfortunately, the terms $\sin(\gamma_\rho T)/\gamma_\rho$   in
\eqref{eq:Explicit1} impede our desired bound since we can only bound them
by $T=\log X$, even under GRH. This leads to the loss of a factor of $\log X$.

To prevent this loss, our next step is to  use Lemma \ref{Partways} for $T$
and $T-a$, with $a>0$ to be selected presently.
\begin{lemma}\label{Mostways} Let $ K/k$ be  an extension of number fields
such that the Riemann Hypothesis holds for $\zeta_K$ and $\zeta_k$. Then, for
$0 < a < T$, we have
\begin{multline} \label{eq:Explicit2}
\abs{
  \paren{ \frac{1}{g(T)} -\frac{1}{g(T-a)} } \log \,\kappa_{K/k}-A(T)+A(T-a)
} \\
\le
(n_K-n_k) a\e^{-(T-a)/2}\beta(T-a) +
c_{a,T} \sum_\rho^{K+k} \frac{1}{ \frac14+\gamma_\rho^2 }
\end{multline}
where the sum $\sum_\rho^{K+k}$ runs over the zeroes of $\zeta_K$ and over
those of $\zeta_k$ (repeating any common zeroes),
\begin{equation}\label{eq:g}
  g(t) := \frac{\e^{-t/2}}{t},\quad
\kappa_{K/k}  :=\lim_{s\to 1}\log\frac{\zeta_K}{\zeta_k}(s),\quad
\end{equation}
$$
A(t):= \sumstack{\p,m\\ \Np^m <\e^t}^{K-k} 
  \frac{\log \Np}{\Np^{m/2}}
  \paren{ \frac{g(m\log \Np)}{g(t)}-1 },
$$
\begin{equation}
c_{a,T} := 1+\frac{a}{4}+\frac{6}{T-a},\quad
\beta(t) := \frac12 \Big(\frac12+\frac1t\Big)e^t
\log\Big( \frac{e^t+1}{e^t-1} \Big).\label{eq:beta}
\end{equation}
\end{lemma}
\begin{proof}
The left-hand side of \eqref{eq:Explicit2} is simply the absolute value of
the difference at $s=1$ of the expressions on left-hand side of
\eqref{eq:Explicit1} for $T$ and $T-a$. Thus, we need to estimate the
difference of terms on the right-hand side of \eqref{eq:Explicit1} for $T$
and $T-a$ at $s=1$. Since all sums in \eqref{eq:Explicit1} are absolutely
convergent, these estimations are straight-forward, but we proceed with the
details.

The mean value theorem  gives
$\abs{ \sin(\gamma T) - \sin(\gamma (T-a) }\le \abs{\gamma a}$ for
$\gamma\in\R$. As GRH  means that $\gamma_\rho\in\R$, we have (using $s=1$,
so $h=\frac12$),
$$
\abs{ 
  -h^2\sum^{K-k}_\rho \frac{\sin(\gamma_\rho T)}{(h^2+\gamma_\rho^2)\gamma_\rho}
-(-h^2)\sum^{K-k}_\rho\frac{\sin \big(\gamma_\rho (T-a)\big)}
                           {(h^2+\gamma_\rho^2)\gamma_\rho}}
\le \frac{a}{4}\sum^{K+k}_{\rho}\frac{1}{\frac14+\gamma_\rho^2 }.
$$
The difference of terms involving $(h+\frac1T)\cos(\gamma_\rho T)$  on the
right-hand side of  \eqref{eq:Explicit1} can be estimated trivially by
$\frac12+\frac1T+\frac12+\frac1{T-a}<1+\frac2{T-a}$. As for the third term,
using $g=g_1$ and $f_{1,X}(t) = \frac{T e^{T/2}}{t e^{t/2}}$, we have
\begin{align*}
\abs{\,2\int_{T}^{+\infty} \frac{ (\frac{t}{2} + 1)}{t^2}  \cos(\gamma_\rho t)
            f_{1,X}(t)\,dt\, }
& \le \frac{2}{T}\int_{T}^{+\infty} \Big(\frac{1}{2} + \frac1t\Big)
            \frac{Te^{T/2}}{te^{t/2}}\,dt = \frac{2}{T},
\end{align*}
where we used \eqref{eq:diffeq} to evaluate the integral.
Applying this with $T$ replaced by  $T-a$, we find that the difference of
the first three sums on the right-hand side of \eqref{eq:Explicit1}
contribute at most $c_a$ times the sums over the zeroes in
\eqref{eq:Explicit2}.

We now consider the difference of the remaining terms on the right-hand side
of \eqref{eq:Explicit1}, \ie those not involving the zeroes $\rho$. Note that
$\frac12L_{K/k}$ simply cancels. We can assume $k\not=K$, for otherwise the
difference vanishes. To control the integrals, abbreviate
$$
 q(T):=
  \int_{T}^{\infty} \frac{1 - f_{1,X}(t)}{2\sinh(t/2)}\,dt,\quad
 \widetilde{q}(T):=\int_{T}^{\infty} \frac{1 - f_{1,X}(t)}{2\cosh(t/2)}\,dt.
$$
Then we have
\begin{align}\label{eq:deriv}
-q^\prime(T)
=\int_T^{\infty}\frac{(1 + \frac{T}{2})e^{T/2}}{t(\e^{t}-1)}\,dt\nonumber
&\le \Big(\frac{1}{2}+ \frac{1}{T}\Big)e^{T/2}\int_T^{\infty}\frac{dt}{\e^t-1}\\
&   = - \Big(\frac{1}{2}+  \frac{1}{T}\Big)\e^{T/2}\log\big(1-\e^{-T}\big).
 \end{align}
Similarly, we have
$$
\big|\widetilde{q}^{\;\prime}(T)\big|\le \Big(\frac{1}{2}+  \frac{1}{T}\Big)\e^{T/2}\log\big(1+\e^{-T}\big).
$$
Moreover, the sign of the derivative  in \eqref{eq:deriv} shows that $q$ and
$\widetilde{q}$ are decreasing functions.

Let  $s_K$ denote the number of complex places of $K$, so that
$n_K=r_K+2s_K$.  Using $k\subset K$ we have $|r_K-r_k|\le n_K-n_k$: indeed,
both sides vanish if $k = K$, and 
$$ -n_K+n_k\le -n_k\le -r_k\le  r_K-r_k = n_K-n_k-2(s_K-s_k)\le n_K-n_k $$
otherwise (the leftmost inequality uses $n_K \geq 2n_k$). Hence
\begin{align*}
&\abs{\frac{n_K-n_k}{2}\big(q(T)-q(T-a)\big) 
  + \frac{r_K-r_k}{2}\big(\widetilde{q}(T)- \widetilde{q}(T-a)\big)}
 \\
& \quad\le \frac{n_K-n_k}{2}\big(q(T-a)-q(T)\big) +  \frac{n_K-n_k}{2}\big(\widetilde{q}(T-a)-\widetilde{q}(T)\big)
\\
& \quad= - \frac{(n_K-n_k)a}{2}\big(  q^\prime(U)+\widetilde{q}^{\;\prime}(U)\big)\qquad(\text{for some }T-a\le U\le T)
\\
& \quad
 \le \frac{(n_K-n_k)a}{2}
 \Big(\frac12+\frac1U\Big)\e^{U/2}
  \big(\log(1+\e^{-U}\big)-\log\big(1-\e^{-U})\big)\\
& \quad
  = (n_K-n_k)a e^{-U/2}\beta(U).
\end{align*}
Since $\beta(U)$ is a decreasing function of $U > 0$, the result
follows from $T-a\le U$.
\end{proof}

Next we give the traditional estimate for the term $\sum_\rho\big(\frac14+\gamma_\rho^2\big)^{-1}$ in Lemma \ref{Mostways}.
%%%
\begin{lemma}[Landau, Stark~\cite{Stark1974}]\label{Estimate} Suppose
$\sigma>1$ and assume the Riemann hypothesis for $\zeta_K$. Then
$$
 \sumstack{\rho\\ \zeta_K(\rho)=0} \frac{1}{\frac14+\gamma_\rho^2}
\le (2\sigma-1) \bigg(\log \Delta_K +\frac{2}{\sigma-1}- d_{K,\sigma}\bigg),
$$
where, letting  $\Psi(\sigma):=\Gamma^\prime(\sigma)/\Gamma(\sigma)$,
\begin{multline} \label{eq:dsigmaK}
d_{K,\sigma}:=-2\frac{\zeta_K^\prime}{\zeta_K}(\sigma)
  +n_K\big(\log(2\pi)-\Psi(\sigma)\big)\\ +
  r_K\frac{\Psi\big(\frac{\sigma+1}{2}\big)-\Psi\big(\frac{\sigma}{2}\big)}{2}-\frac{2}{\sigma}.
 \end{multline}
\end{lemma}
\begin{proof}
For $h:=\sigma-\frac12>\frac12$ and $\gamma\in\R$, we have
$$
\frac{1}{\frac14+\gamma^2}=\frac{4h^2}{h^2+(2h\gamma )^2}<\frac{4h^2}{h^2+ \gamma^2}.
$$
Now, since $\sigma\in\R$ and the zeroes $\rho=\frac12+i\gamma_\rho$ come in
conjugate pairs,
$$
\sum_\rho\frac{h}{h^2+\gamma_\rho^2}=\sum_\rho\re\Big(\frac{1}{\sigma-\rho}\Big)=
\re\Big(\sum_\rho\frac{1}{\sigma-\rho}\Big)= \sum_\rho\frac{1}{\sigma-\rho},
$$
where the latter sums are understood as $\lim_{R\to\infty}\sum_{|\gamma_\rho|<R}(\sigma-\rho)^{-1}$.
This sum was evaluated by Stark \cite[eq.~(9)]{Stark1974}
(cf.~\cite[Satz 180]{Lan:einf}). Namely,\footnote{\ One can prove
\eqref{eq:Stark} with the explicit formula, using
$F(x):=\exp(-(\sigma-\frac12)|x|)$. However, the classical proof in
\cite{Stark1974} with the Weierstra{\ss} product and functional equation is
faster.}
\begin{equation}\label{eq:Stark}
\sum_\rho\frac{1}{\sigma-\rho}=\frac{\log \Delta_K}{2}+
\frac{1}{\sigma-1}+\frac{1}{\sigma} - \frac{1}{2} d_{K,\sigma},
\end{equation}
where we have used the duplication formula
$$\Psi(\sigma) = \log2
  + \frac{ \Psi(\frac{\sigma}{2}) + \Psi(\frac{\sigma+1}{2}) }{2}.
$$
\end{proof}

\begin{proof}[Proof of Theorem~\ref{Mainresult}]
In Lemma \ref{Mostways}, take $k=\Q$, $a:=\log(9)$ and $T:=\log{X}$.
The hypothesis $0 < a < T$ in Lemma \ref{Mostways} is satisfied since $X>9$.
A short calculation shows
\begin{align}\label{eq:Diff}
 \frac{1}{g(T)} -\frac{1}{g(T-a)}  =\frac{2\sqrt{X}\log(3X)}{3},
  \end{align}
with $g$ as in \eqref{eq:g}.
Since  $\kappa_{K/\Q}=\kappa_K$ and 
$$A(T) - A(T-a)=B_K(X) - B_K(X/9),$$
Lemma~\ref{Mostways} yields for any $\sigma>1$,
\begin{multline}\label{eq:Step1}
  \frac{2\sqrt{X}\log(3X)}{3}\abs{\log \,\kappa_K- f_K(X)} \\
  \le  c_{a,T}\sum^{K+\Q}_{\rho} \frac{1}{ \frac14 + \gamma_\rho^2 }
  +  (n_K-n_\Q) a\e^{-(T-a)/2}\beta(T-a).
\end{multline}
The sum over the nontrivial zeroes of $\zeta_\Q$ is classical
\cite[\S12, eqs.~(10) and (11)]{Dav:MNT},
$$
\sumstack{\rho\\ \zeta_\Q(\rho)=0}\frac{1}{ \frac14+\gamma_\rho^2 }=\frac{C}{2}+1-\frac{\log(4\pi)}{2}=.023095\cdots.
$$
We also have
$$ c_{a,T}= 1+\frac{\log 9}{4}+\frac{6}{\log (X/9)}, $$
$$ (n_K-n_\Q) a\e^{-(T-a)/2}=\frac{(n-1)3\log 9 }{\sqrt{X}}. $$
We have already noted in the proof of Lemma~\ref{Mostways} that $\beta(T-a) =
\beta\big(\log(X/9)\big)$ is a decreasing function of $X$, for $X > 9$.
Moreover, $\beta\big(\log(X/9)\big) < 1$ for $X > 68.1$.

 We turn to Lemma \ref{Estimate} to bound the sum over the zeroes of
$\zeta_K$. The  main term  in that lemma (say for $1<\sigma<3$) is
$$(2\sigma-1)\Big(\log \Delta_K+\frac{2}{\sigma-1}\Big).$$
This is minimized when $\sigma:=1+\paren{\log \Delta_K}^{-\frac12}$.
We fix  this value of $\sigma$ for the rest of this proof. Then
$$
(2\sigma-1)\Big(\log \Delta_K + \frac{2}{\sigma-1}\Big)
  =\Big(\sqrt{\log \Delta_K}+2\Big)^2.
$$
Since $\Delta_K\ge 3$ for $K\neq\Q$, we have $1<\sigma\le 1+\big(\log3\,\big)^{-\frac12}<3$. 

We now estimate $d_{K,\sigma}$ in Lemma \ref{Estimate}. Since
$\frac{\zeta_K^\prime}{\zeta_K}(\sigma)<0$, $n_K\geq 2$ and $\Psi(x)$ is
increasing for $x>0$, we have
\begin{align*}
d_{K,\sigma}
  & > 2\big(\log(2\pi)-\Psi(\sigma)\big) -\frac{2}{\sigma}\qquad\text{(since
$\log(2\pi) - \Psi(3) > 0$)} \\
  & =2 \log(2\pi)-2
  \Psi(\sigma+1) > 2 \log(2\pi)-2 \Psi(4)=1.163\cdots
\end{align*}
Since $2\sigma-1>1$, it follows that
$$
\sumstack{\rho\\ \zeta_\Q(\rho)=0}\frac{1}{ \frac14+\gamma_\rho^2 }
-(2\sigma-1)d_{K,\sigma} < 0.
$$
Hence \eqref{eq:Step1} and Lemma \ref{Estimate} give, for $X \geq 68.1$,
\begin{multline*}
\frac{2\sqrt{X}\log(3X)}{3} \abs{\log \kappa_K- f_K(X)} \\
\leq
  \Big(1 + \frac{\log 9}{4}+
  \frac{6}{\log(X/9)}\Big) \Big({\sqrt{\log \Delta_K}+2}\Big)^2
  + \frac{(n-1)3 \log 9}{\sqrt{X}}.
\end{multline*}
 Pulling out a factor of $(1 + \frac{\log 9}{4})\log
\Delta_K$ gives Theorem \ref{Mainresult}, since
$$\frac{3}{2}\paren{1 + \frac{\log 9}{4}} < 2.324,\quad
  \frac{6}{1 + \frac{\log 9}{4}} < 3.88,\quad
  \frac{3\log 9}{1 + \frac{\log 9}{4}} < 4.26.$$
\end{proof}
An examination of the proof shows that the choice of $a=\log 9 =2.197\cdots$
is only nearly  optimal. The optimal $a\approx3.01$  improves the constant
2.324 in Theorem \ref{Mainresult} to about 2.253. We have chosen $a=\log9$
because it simplifies several expressions, beginning with  \eqref{eq:Diff}.

\begin{remark}\label{rem:Mainresult}
Although the proof requires $X> e^a = 9$, the restriction $X\geq 69$ was only
needed to ensure $\beta\big(\log (X/9)\big) < 1$. The conclusion of
Theorem~\ref{Mainresult} holds for $X > 9$ provided the final term
$\frac{4.26(n-1)}{\sqrt{X}\log \Delta_K}$ is replaced by
$\frac{4.26(n-1)\beta(\log (X/9))}{\sqrt{X}\log \Delta_K}$.
\end{remark}

We can improve slightly on Theorem \ref{Mainresult} by not dropping some
favorable terms.
\begin{theorem}\label{Mainresult2} Let $K$ be a number field of degree
$n_K$ with $r_K$ real places. With the same assumptions and notation as
in Theorem~\ref{Mainresult}, except we now only assume $X > 9$,
we have for any $\sigma>1$
\begin{multline*}
   \abs{\log \kappa_K -  f_K(X)}
  \le  \frac{2.324(2\sigma-1)}{\sqrt{X}\log(3X)} \\
\cdot  \Bigg(
   \delta(K,\sigma,X)
   \Big(1+ \frac{3.88}{\log(X/9)}\Big)
  + 
   \frac{4.26(n_K-1)\beta\big(\log (X/9)\big)}{(2\sigma-1)\sqrt{X}}
\Bigg),
\end{multline*}
where
\begin{multline*}
  \delta(K,\sigma,X):=\log \Delta_K 
  +\frac{\frac{C}{2}+1-\frac{\log(4\pi)}{2}}{2\sigma-1}
  + \frac{2}{\sigma-1}
  + \frac{2}{\sigma}
  -  2\sumstack{\p\\ \Np<X}
         \frac{\log\Np}{\Np^\sigma-1} \\
  - n_K\big(\log(2\pi)-\Psi(\sigma)\big) 
  -
  r_K\frac{\Psi\big(\frac{\sigma+1}{2}\big)-\Psi\big(\frac{\sigma}{2}\big)}{2}.
\end{multline*}
Here $\Psi(\sigma):=\Gamma^\prime(\sigma)/\Gamma(\sigma)$
and $\beta(t)$ is defined in \eqref{eq:beta}.
\end{theorem}
\begin{proof}
We proceed as in the proof of the previous theorem, fixing again $a=\log 9$, but we do not fix $\sigma$. If in
 $ d_{K,\sigma}$   $\big($see \eqref{eq:dsigmaK}$\big)$
we truncate
$-\frac{\zeta^\prime_K}{\zeta_K}(\sigma)= \sum_{ \p}\frac{\log\Np}{\Np^\sigma-1}$, instead of \eqref{eq:Step1} we obtain
\begin{multline*}
 \abs{\log \kappa_K -  f_K(X)}
  \leq \frac{3(2\sigma-1)}{2\sqrt{X}\log(3X)}  \\ 
\cdot \Bigg(\!
  \delta(K,\sigma,X)
  \Big(1 + \frac{\log 9}{4}+ \frac{6}{\log(X/9)}\Big)
  + \frac{(n_K-1)3\log 9\cdot \beta\big(\log(X/9)\big)}{(2\sigma-1)\sqrt{X}}
\!\Bigg).
\end{multline*}
\end{proof}
In Theorem \ref{Mainresult2}, $\sigma=1+1/\sqrt{\log\Delta_K}$ is usually a
good choice. Taking instead $\sigma = 1.5$, the value used by
Bach~\cite[Lemma 4.2]{Bac:eul}, we obtain
\begin{corollary} With the same assumptions and notation as in
  Theorem~\ref{Mainresult2},  we have
\begin{multline*}
 \abs{\log \kappa_K -  f_K(X)}
 \le \frac{4.65}{\sqrt{X}\log(3X)}\Bigg(
   \frac{2.23n_K\beta\big(\!\log(X/9)\big)}{\sqrt{X}} \\ 
 + \Big(1+ \frac{3.88}{\log(X/9)}\Big)
   \Big( \log \Delta_K + 3.35 - 1.801n_K - .619r_K 
         -2\sumstack{\p\\ \Np<X} \frac{\log\Np}{\Np^{1.5}}\Big)
\Bigg).
\end{multline*}
\end{corollary}

\section{Examples}\label{se:examples}
This section compares experimentally three algorithms evaluating $\log
\kappa_K$, by splitting rational primes up to a fixed bound $X$. These
involve the functions $A_K(X)$, $g_K(X)$ and $f_K(X)$ defined below. All
programs were implemented in the PARI/GP system~\cite{PARI}.

We first define Schoof's approximation
$$ A_K(X) := \log \prod_{p < X} \frac{1 - p^{-1}}
                            {\prod_{\p\mid p, \Np < X} 1 - \Np^{-1}},$$
originating in \cite{Sch:cla} and whose distance to $\log \kappa_K$ is bounded
by Bach~\cite[Theorem~6.2 and Table~2]{Bac:eul} under GRH.\footnote{As Bach
warns, this bound assumes a result of Oesterl\'e's \cite[equation
(12)]{Bac:eul} whose proof has never been published.} It is in principle
weaker than our $f_K(X)$ or Bach's $g_K(X)$, since it only satisfies
$$
  \abs{ \log \,\kappa_K - A_K(X) } \ll \frac{\log \Delta_K}{\sqrt{X}}
$$
(see also the remark at the end of \cite[\S8]{Bac:eul}).
For $X$ even, Bach's approximation to $\log \kappa_K$ is
$$g_K(X) := \sum_{i = 0}^{x - 1} a_i A_K(x+i),$$
where $x = X/2$, and
$$a_i := \frac{(x+i)\log(x+i)}{\sum_{j = 0}^{x-1} (x+j)\log (x+j)}.$$
The distance $\abs{g_K(X) - \log \kappa_K}$ is bounded in~\cite[Theorem~6.3
and Table~1]{Bac:eul}, assuming GRH. Finally, our function
$$
f_K(X):= \frac{3\big(B_K(X)-B_K(X/9)\big)}{2\sqrt{X}\log(3X)}
$$
appears in~Theorem~\ref{Mainresult}. We assume $X\geq 10$ and include the
term $\beta\big(\!\log (X/9)\big)$ from Remark~\ref{rem:Mainresult} in the
error bound.

\begin{table}
\caption{Least $X$ so that $\abs{f_K(X) - \log \kappa_K} < \frac{1}{2}\log 2$}
\begin{tabular}{|l|c|c|c|c|c|}\hline
$\Delta$ & $n = 2$
   & $n = 6$
   & $n = 10$
   & $n = 20$
   & $n = 50$\\\hline
$10^{5}$ & 
     1,619 & 
     1,632 & 
-- &
-- &
--
\\\hline
$10^{10}$ & 
     3,169 & 
     3,181 & 
     3,194 & 
-- &
--
\\\hline
$10^{20}$ & 
     6,838 & 
     6,850 & 
     6,861 & 
-- &
--
\\\hline
$10^{50}$ & 
     21,619 & 
     21,629 & 
     21,639 & 
     21,665 & 
--
\\\hline
$10^{100}$ & 
     56,332 & 
     56,341 & 
     56,351 & 
     56,374 & 
     56,445\\\hline
$10^{200}$ & 
     156,151 & 
     156,160 & 
     156,169 & 
     156,191 & 
     156,256\\\hline
\end{tabular}
\vskip 0.1cm
\end{table}

\begin{table}
\caption{Least $X$ so that $\abs{g_K(X) - \log \kappa_K} < \frac{1}{2}\log 2$}
\begin{tabular}{|l|c|c|c|c|c|}\hline
$\Delta$ & $n = 2$
   & $n = 6$
   & $n = 10$
   & $n = 20$
   & $n = 50$\\\hline
$10^{5}$ & 
     4,469 & 
     6,493 & 
-- &
-- &
--
\\\hline
$10^{10}$ & 
     9,799 & 
     11,324 & 
     13,857 & 
-- &
--
\\\hline
$10^{20}$ & 
     22,476 & 
     25,621 & 
     28,935 & 
-- &
--
\\\hline
$10^{50}$ & 
     91,044 & 
     96,596 & 
     99,999 & 
     110,802 & 
--
\\\hline
$10^{100}$ & 
     268,680 & 
     276,338 & 
     284,088 & 
     303,864 & 
     366,575\\\hline
$10^{200}$ & 
     866,110 & 
     878,749 & 
     891,468 & 
     923,610 & 
     1,000,000\\\hline
\end{tabular}
\vskip 0.3cm
\end{table}

\begin{table}
\caption{Least $X$ so that $\abs{A_K(X) - \log \kappa_K} < \frac{1}{2}\log 2$}
\begin{tabular}{|l|c|c|c|c|c|}\hline
$\Delta$ & $n = 2$
   & $n = 6$
   & $n = 10$
   & $n = 20$
   & $n = 50$\\\hline
$10^{5}$ & 
     13,420 & 
     46,329 & 
-- &
-- &
--
\\\hline
$10^{10}$ & 
     31,829 & 
     65,465 & 
     119,149 & 
-- &
--
\\\hline
$10^{20}$ & 
     76,617 & 
     130,922 & 
     212,428 & 
-- &
--
\\\hline
$10^{50}$ & 
     347,503 & 
     476,196 & 
     566,686 & 
     1,000,001 & 
--
\\\hline
$10^{100}$ & 
     1,080,396 & 
     1,298,034 & 
     1,541,474 & 
     2,268,510 & 
     5,559,680\\\hline
$10^{200}$ & 
     4,054,695 & 
     4,502,259 & 
     4,979,474 & 
     6,305,841 & 
     11,493,924\\\hline
\end{tabular}
\end{table}

We evaluate these three functions by first splitting all primes $p\leq X$,
and then by using $O(X)$ elementary operations in
$\{+,\times,/,\log,\sqrt{\cdot}\}$. We can thus approximate those functions
at $X$ to a fixed accuracy in time $\Ot(X)$, softly linear in $X$. The
application to Buchmann's algorithm requires the computation of $\log
\kappa_K$ with an error bounded by $\frac{1}{2}\log 2$.

For each function $h\in \{ f_K, g_K, A_K\}$, given a bound of the number
field degree $n_K \leq n$ and discriminant $\Delta_K\leq \Delta$, Tables 1, 2
and 3 list the first integer $X$ such that
$$\abs{ h(X) - \log \kappa_K } < \frac12 \log 2,$$
according to the error bounds mentioned above (all of which assume GRH).
A dash (--) indicates that this value of $n_K$ and $\Delta_K$
is forbidden by Odlyzko's discriminant bounds~\cite[Table 1]{Martinet1982}.

Besides the asymptotic improvement for large discriminants, the weak
dependency on the number field degree in secondary error terms makes our bound
almost impervious to the degree, while Bach's and Schoof's are noticeably
affected by $n$, even for relatively large discriminants.

\bibliographystyle{amsplain}
\providecommand{\bysame}{\leavevmode\hbox to3em{\hrulefill}\thinspace}
\providecommand{\MR}{\relax\ifhmode\unskip\space\fi MR }
\providecommand{\MRhref}[2]{%
  \href{http://www.ams.org/mathscinet-getitem?mr=#1}{#2}
}
\providecommand{\href}[2]{#2}

\end{document}